\documentclass[12pt]{amsart}
\usepackage{amsmath}
\usepackage{a4wide}
\usepackage{latexsym, amsmath, amsfonts, amsthm, amssymb}

\usepackage[latin1]{inputenc}
\usepackage{graphicx}
\usepackage{mathrsfs}

\newtheorem{theo}{Theorem}

\newtheorem{cor}[theo]{Corollary}
\newtheorem{rem}[theo]{Remark}

\newcommand{\zero}{\mathbf{0}}

\newcommand{\R}{\ensuremath{\mathbb{R}}} 

\newcommand{\upchi}{\raise1pt\hbox{$\chi$}}
\def\dom{\textrm{dom}}

\newcommand{\be}{\begin{equation}}
\newcommand{\ee}{\end{equation}}
\def\benu{\begin{enumerate}}
\def\eenu{\end{enumerate}}

\begin{document}

\title[Regularity of monotone transport maps between unbounded domains]{Regularity of monotone transport maps\\ between unbounded domains}

\author{Dario Cordero-Erausquin and Alessio Figalli}

\thanks{
Institut de Mathematiques de Jussieu, Sorbonne Universit\'e - UPMC (Paris 6), France\\  ({\tt
      dario.cordero@imj-prg.fr})}
\thanks{ETH Z\"urich, Mathematics Department, R\"amistrasse 101, 8092 Z\"urich, Switzerland\\
({\tt
      alessio.figalli@math.ethz.ch})}

\begin{abstract}
The regularity of monotone transport maps plays an important role in several applications to PDE and geometry.
Unfortunately, the classical statements on this subject are restricted to the case when the measures are compactly supported. In this note we show that, in several situations of interest, one can to ensure the regularity of monotone maps even if the measures may have unbounded supports.
\end{abstract}

 \dedicatory{A Luis A. Caffarelli en su 70 a\~nos, con  amistad y admiraci\'on}

\maketitle

Given two Borel probability measures $\mu$ and $\nu$ on $\R^n$, with $\mu$ absolutely continuous with respect to the Lebesgue measures, it is known~\cite{mccann95} that there exists a lower semicontinuous convex function $u:\R^n \to \R\cup\{+\infty\}$ such that $T:=\nabla u$ pushes-forward $\mu$ onto $\nu$. We call this monotone map $T=\nabla u$ the \emph{Brenier-McCann} map. 
Although not needed, we recall that this map corresponds to the (unique {a.e.}) optimal transport map for the quadratic cost pushing  forward $\mu$ onto $\nu$ (see for instance \cite[Section 3]{dPF2}).

When $\mu$ and $\nu$ have densities $F$ and $G$ respectively (namely, $d\mu(x)=F(x)dx$ and $d\nu(y)=G(y)dy$), it follows from the work of McCann~\cite{mccann} that $u$ is finite inside the interior of the support of $F$ and the Monge-Amp\`ere equation 
\begin{equation}\label{mamc}
F(x) = G(\nabla u(x) )\det D^2 u(x)
\end{equation}
is satisfied in the following weak sense: the convex function $u$ admits at a.e. point a second derivative $D^2 u$ (that coincides also with the absolutely continuous 
part of the distributional second derivative), and with this second derivative the equation~\eqref{mamc} is satisfied $\mu$-a.e.

Alternatively, there exists a stronger regularity theory, first established by Caffarelli,  ensuring that~\eqref{mamc} holds in a classical sense (in particular the second distributional derivatives have no singular part, and $D^2u$ is defined with no ambiguity).  More precisely, 
most of the classical regularity of solutions of the Monge-Amp\`ere equation  rely on the study of Alexandrov solutions (see \eqref{eq:Alex} below for a definition), as in the seminal papers~\cite{caf1,cafC1a}. Hence, to apply this theory to monotone maps, the main difficulty is to verify that the Brenier-McCann map provides an Alexandrov solution to~\eqref{mamc},
 at least under some suitable assumptions on the measures (essentially, one has to assume that the target measure has a convex support \cite{caf2}).

We recall that the Brenier-McCann map is a handy tool to prove geometric and functional inequalities.  However, a technical difficulty arises when one would like to ensure that this map is sufficiently smooth, since that requires several restrictive assumptions on the densities and the supports.
While some are necessary (namely, the convexity of the support of the target measure, see for instance \cite{FRV-necsuff}), several other assumptions are of purely technical nature. Indeed, one typically assumes that the supports are \emph{bounded}, and that the densities are (globally) bounded away from zero and infinity on their respective support. These restrictions are due to the way the theory was built, as a part of the analysis of nonlinear PDE, rather than from the point of view of monotone transportation. Since in applications we often want to transport densities that are defined on the whole $\R^n$, this is a serious restriction.

The goal of this note is to:
\begin{enumerate}
\item allow for unbounded domains;
\item allow for local (i.e., on every compact set) boundedness of the densities, away from zero and infinity. 
\end{enumerate}
Here and in the sequel, given a Borel set $E$, $|E|$ stands for its Lebesgue measure and $1_E$ denotes its characteristic function (namely, $1_E(x)=1$ if $x\in E$ and $1_E(x)=0$ if $x \not\in E$).
This is our main result:

\begin{theo}\label{maintheo}
Let $X$ and $Y$ be two open sets in $\R^n$,  and assume that $|\partial X|=0$
and that $Y$ is convex. Let $F,G:\R^n\to \R$ be two nonnegative Borel functions with $\int_X F = \int_Y G= 1$ such that $F,1/F\in L^\infty_{\rm loc}(X)$ and $G, 1/G\in L^\infty(Y\cap B_R)$ for any $R>0$. Denote by $T=\nabla u$ the Brenier-McCann map  pushing $1_X(x) F(x)dx$ forward to $1_Y(y)G(y)dy$. 
Assume in addition that at least one of these assumptions hold:
\begin{enumerate}
\item[(1)] either $n=2$;
\item[(2)] or $n\geq 3$, $F,1/F\in L^\infty(X\cap B_R)$ for any $R>0$,
and one of the following holds:
\begin{enumerate}
\item[(2-a)]
$X=\R^n$;
\item[(2-b)] $Y$ is bounded;
\item[(2-c)] $X$ is locally uniformly convex.
\end{enumerate}
\end{enumerate}
 Then:
\begin{enumerate}
\item[(i)] For any $X'\subset\subset X$ there exist $\alpha,\epsilon>0$ such that $u \in C^{1,\alpha}(X')\cap W^{2,1+\epsilon}(X')$. In addition
$T=\nabla u:X\to Y$ is continuous,
$T(X)$ is an open subset of $Y$ of full measure (inside $Y$), and
$T:X\to T(X)$ is a homeomorphism.
Finally, if $X=\R^n$ then $T(X)=Y$.
\item[(ii)] If in addition there exist $k\geq0$ and $\beta \in (0,1)$ such that $F$ and $G$ are of class $C^{k,\beta}_{\rm loc}$ on $X$ and $Y$ respectively, then
$T=\nabla u:X\to T(X)$ is a diffeomorphism of class $C^{k+1,\beta}_{\rm loc}$.
\end{enumerate}
\end{theo}

\begin{rem}{\em 
The case $X=\R^n$ and $Y$ bounded was already considered in~\cite{ADM}.}
\end{rem}

\begin{rem}\label{rmk:unif convex}{\em 
In Case (2-c), the assumption $X$ being locally uniformly convex can be relaxed. More precisely, it suffices to ask that near each point on $\partial X$ there exists a system of coordinates such that $\partial X$ can be locally written as the graph of a function $f:\R^{n-1}\to \R$ with $f(0)=0$
and $f(x')\geq c|x'|^{\lambda}$ for some $c>0$ and $0<\lambda<\frac{2(n-1)}{n-2}$, see Remark \ref{rmk:bdry}.}
\end{rem}

\begin{rem}{\em 
\label{rmk:duality}
It is well known that if $S$ denotes the optimal transport map from $\nu$ to $\mu$, then $S=T^{-1}=\nabla u^\ast$ $\nu$-a.e., where $u^\ast$  is the Legendre transform of $u$ (see for instance \cite[Section 6.2.3]{AGS}).
Hence, if the assumptions on $X$ and $Y$ are reversed, we deduce that $S:Y\to S(Y)\subset X$ is a homeomorphism (resp. a diffeomorphism if the densities are smoother).
Hence, since $S=T^{-1}$, $T$ is a homeomorphism (resp. diffeomorphism)
from $S(Y)$ onto $Y.$
In particular, whenever one can apply Theorem \ref{maintheo} both from $X$ to $Y$ and from $Y$ to $X$, then one deduces that
$$
T:X\to T(X)\subset Y\quad \text{and}\quad
T^{-1}:Y\to T^{-1}(Y)\subset X \quad \text{are homeomorphisms}, 
$$
thus $T(X)=Y$.
}
\end{rem}

An immediate consequence of Remark \ref{rmk:duality} is the following corollary for the case $X$ convex and $Y=\R^n$:
\begin{cor}
\label{maincor}
Let $X$ be an open  convex set in $\R^n$ (in particular, one may take $X=\R^n$), let $F,G:\R^n\to \R$ be two nonnegative Borel functions with $\int_X F = \int_{\R^n} G= 1$ such that $F,1/F$ belong to $L^\infty(X\cap B_R)$ and $G, 1/G$ belong to $L^\infty(B_R)$ for any $R>0$. Denote by $T=\nabla u$ the Brenier-McCann map  pushing $1_X(x) F(x)dx$ forward to $G(y)dy$.
Then $T:X\to \R^n$ is a homeomorphism. Also, if there exist $k\geq 0$ and $\beta \in (0,1)$ such that $F$ and $G$ are of class $C^{k,\beta}_{\rm loc}$, then $T:X\to \R^n$ is a diffeomorphism of class $C^{k+1,\beta}_{\rm loc}$.
\end{cor}

\begin{rem}{\em 
To illustrate some subtleties behind the proof of Theorem \ref{maintheo},
in the appendix we provide an example showing that
Corollary \ref{maincor} would be false for $n\geq 3$ if we replace the assumption $F, 1/F\in L^\infty(X\cap B_R)$ for any $R>0$ with the weaker hypothesis $F, 1/F\in L^\infty_{\rm loc}(X)$.}
\end{rem}

Before proving Theorem \ref{maintheo}, let us recall some background and fix some notation. \\

 Given a convex function $u:\R^n\to \R\cup\{+\infty\}$, we denote by $\dom(u)$ the convex set where $u$ is finite. Recall that $u$ is locally Lipschitz in the interior of $\dom(u)$, so in particular it is differentiable a.e. there (see for instance \cite[Appendix A.4]{figalliBook}). We denote by $\dom(\nabla u)\subset \dom(u)$ the set where $u$ is differentiable.  

 As we mentioned before, the regularity  theory builds upon what are called \emph{Alexandrov solutions}, that we now introduce. Given a convex function $u:\R^n \to \R\cup\{+\infty\}$,
 one can associate to it the so-called \emph{Monge-Amp\`ere measure} $\mathscr{M}_u$ defined by 
$$\mathscr{M}_u(A) := |\partial u (A)|\qquad \text{for every Borel set $A\subset \R^n$,}$$ 
where
$$
\partial u(A):=\bigcup_{x \in A}\partial u(x)
$$
and $\partial u(x)$ denotes the subdifferential of $u$ at $x$, that is
$$
\partial u(x):=\{p\in \R^n\,:\,u(z)\geq u(x)+p\cdot (z-x)\quad \forall\,z \in \R^n\}
$$
(see for instance \cite[Section 2.1]{figalliBook} for a proof of the fact that $\mathscr M_u$ is a Borel measure).
When $u$ is of class $C^2$, one can prove that the Monge-Amp\`ere measure of $u$ is simply given by $\det D^2 u\,dx$ (see \cite[Example 2.2]{figalliBook}).

Comparing this with \eqref{mamc}, one says that the Brenier-McCann map $T=\nabla u$ is an Alexandrov solution
of the Monge-Amp\`ere equation if  
\begin{equation}
\label{eq:Alex}
\mathscr{M}_u(A) = \int_{A} \frac{F(x)}{G(\nabla u(x))}\, dx\qquad \text{for all $A\subset X$ Borel.}
\end{equation}
We recall that, when $\nabla u$ is the Brenier-McCann map, then $u$ is finite inside $X$. In particular $u$ is differentiable a.e. in $X$ and the above definition makes sense.

\begin{proof}[Proof of Theorem~\ref{maintheo}]
Let $\mu$ and $\nu$ be Borel probability measures on $\R^n$ with densities $1_X(x) F(x)$ and $1_Y(y)G(y)$, respectively. 
Observe that the assumptions on the densities and on the domains ensure that  $\mu$ is equivalent to the Lebesgue measure on $\overline X$, in the sense that these measures share the same sets of zero measure, and that $\nu$ is equivalent to the Lebesgue measure on $\overline Y$.  
\\

\noindent
{\it - Step 1: extending $u$ outside $X$.}
Note that,
from the point of view of mass transport, one only cares about the functions $u$ and $T=\nabla u$ inside the set $X$, since there is no mass in $\R^n\setminus X$.
However, for technical purposes, it is convenience to fix a precise definition for $u$ also outside $X$.
Although there is no unique way of defining $u$ outside $X$, in our case it is convenient to use the ``minimal'' possible extension.
More precisely, since $u$ is convex and finite inside $X$, for any $z \in X$ we have
\begin{equation}\label{extension}
u(z) = \sup_{x\in X,\,  p \in \partial u (x)} \big\{ u(x) + p\cdot (z-x)\big\}.
\end{equation}
Then we take~\eqref{extension} as definition of $u(z)$ for $z\in \R^n \setminus X$.

It is worth mentioning that this does not change $\partial u$ (and $\nabla u$) inside $X$, since the subdifferential at points in the interior of $\dom (u)$ is determined only by the local behavior of $u$ (compare with \cite[Proof of Theorem 4.6.2]{figalliBook}).
\\

\noindent
{\it - Step 2: some useful observations.}
Since $u$ is finite inside $X$, it is locally Lipschitz there. Thus $$|X\setminus \dom (\nabla u)|= \mu(X\setminus \dom (\nabla u))=0.$$ Moreover, since $\mu((\nabla u)^{-1}(\R^n \setminus Y))=\nu(\R^n \setminus Y)=0$, there exists a Borel set $N$ satisfying 
\begin{equation}\label{gradimage}
N\subset X, \qquad \mu(N)=0, \qquad \textrm{and }\qquad  \nabla u(x)\in Y\quad
\forall \,x\in (X\setminus N) \cap \dom(\nabla u).
\end{equation}
Another useful observation is that  $Y$ is contained in the domain of the Legendre's transform $u^\ast:\R^n\to \R\cup\{+\infty\}$ of $u$, given by 
$$
u^\ast(y):= \sup_{x\in \R^n}\bigl\{ {x\cdot y} - u(x) \bigr\}.
$$ 
To see this, we can for instance invoke the fact that $\nabla u^\ast$ is the Brenier-McCann map between $\nu$ and $\mu$ (see for instance \cite[Section 6.2.3]{AGS}), hence $\nu(Y\setminus \dom (\nabla u^\ast))=0$.
\\

\noindent
{\it - Step 3: $u$ is an Alexandrov solution.}
We now want to prove that $u$ satisfies \eqref{eq:Alex}.
Actually, we shall prove more, namely that
\begin{equation}
\label{eq:Alex2}
\mathscr{M}_u(A) = \int_{A\cap X} \frac{F(x)}{G(\nabla u(x))}\, dx\qquad \text{for all $A\subset \R^n$ Borel.}
\end{equation}
This will be important in the next steps of the proof.

To prove \eqref{eq:Alex2}, we suitably modify the argument of Caffarelli~\cite{caf2} as presented in~\cite{dPF2}. We need the following two observations.

First, given a Borel set  $A\subset X$, there exists a Borel set $N_A\subset X$ with $\mu(N_A)=0$ such that 
\begin{equation}\label{aeinverse}
 (\nabla u)^{-1} (\partial u (A)\cap Y)\cap X = (A\setminus N) \cup N_A .
 \end{equation}
Indeed,  recalling~\eqref{gradimage},
we have the inclusion
$$A\setminus N \subset  (\nabla u)^{-1} (\partial u (A)\cap Y).$$
In addition, recalling that $u^\ast$ denotes the Legendre's transform of $u$, the set $\big[(\nabla u)^{-1} (\partial u (A)\cap Y)\cap X\big]\setminus (A\setminus N) $ is contained in 
\begin{multline*}
N\cup\{x\in X\cap \dom(\nabla u)\,:\,  \nabla u(x) \in \partial u(A)\cap Y \textrm{ and } x \notin A\ \} \\
 \subset N\cup  \{x\in X\cap \dom(\nabla u)\, :\,  u^\ast \textrm{ not differentiable at } \nabla u(x), \textrm{ and } \nabla u(x) \in  Y\}
\end{multline*}
(see \cite[Appendix A.4]{figalliBook} for more details).
Hence, since the set $$\{y\in  Y\, :\, u^\ast \textrm{ not differentiable at } y\}$$  has $\nu$-measure zero inside $Y$ (see Step 2), we conclude that its pre-image by $\nabla u$ is of $\mu$-measure zero, as desired.

Second, we need the following crucial property, first obtained by  Caffarelli in~\cite[Lemma~1]{caf2} for bounded domains (this is where the convexity of $Y$ comes into play):
\begin{equation}\label{subgradimage}
\partial u (X) \subset \overline Y\qquad \text{and} \qquad |Y\setminus \partial u (X)|=0.
\end{equation}
To prove this fact it is enough to observe that, thanks to \eqref{gradimage}, $\nabla u(x)\in Y$ for a.e. $x \in X$. By continuity of the subdifferential of convex functions (see \cite[Lemma A.4.8]{figalliBook}) we deduce that
$$\nabla u(X\cap {\rm dom}(\nabla u))\subset \overline Y,$$ hence the inclusion $\partial u (X) \subset \overline Y$ follows by \cite[Corollary A.4.11]{figalliBook}. To prove the second part of \eqref{subgradimage}, we note that $(\nabla u)_\#\mu=\nu$, and $Y$ coincides with the support of $d\nu(y)=1_Y(y)G(y)\,dy$. Thus, for a.e. $y \in Y$
there exists $x \in X\cap {\rm dom}(\nabla u)$ such that $y=\nabla u(x)$. This proves that $\nabla u(X\cap {\rm dom}(\nabla u))$ has full measure inside $Y$, which concludes the proof of \eqref{subgradimage}.

Actually, we note that by the way $u$ is defined in the whole $\R^n$ (see \eqref{extension})
it follows that 
$$
\partial u(\R^n)\subset \overline{\partial u(X)}\subset \overline Y,
$$
that combined with \eqref{subgradimage} implies that 
$$
\partial u(\R^n\setminus X)\subset \bigl(\partial u(\R^n\setminus X)\cap \partial u(X)\bigr)\cup \hat N,
$$
where $\hat N$ is a set of Lebesgue measure zero. Since also $\bigl(\partial u(\R^n\setminus X)\cap \partial u(X)$ has measure zero (see \cite[Lemma A.30]{figalliBook}), we conclude that
\begin{equation}\label{MA outside X}
\mathscr{M}_u(\R^n\setminus X)=0.
\end{equation}

Using~\eqref{aeinverse} and~\eqref{subgradimage}, given a Borel set $A\subset X$ and a set $N_A$ of $\mu$-measure zero as in \eqref{aeinverse}, since $\partial u (A) \subset \overline Y$ and $(\nabla u)_\#\mu=\nu$, we can write
\begin{multline*}
|\partial u(A) |  = |\partial u(A) \cap Y| =\int_{\partial u (A) \cap Y} \frac{1}{G(y)} \, d\nu(y) = \int_{(\nabla u)^{-1} (\partial u (A)\cap Y)\cap X} \frac1{G(\nabla u(x))} \, d\mu(x) \\
= \int_{(A\setminus N) \cup N_A} \frac1{G(\nabla u(x))} \, d\mu(x) =  \int_{A} \frac1{G(\nabla u(x))} \, d\mu(x) = \int_{A} \frac{F(x)}{G(\nabla u(x))} \, dx.
\end{multline*}
Since $\mathscr M_u$ is a Borel measure that vanishes outside $X$ (see \eqref{MA outside X})
this proves \eqref{eq:Alex2}.
\\

\noindent
{\it - Step 4: $u$ is strictly convex inside $X$.}
In order to apply the regularity theory for Alexandrov solutions, one needs to prove that $u$ is strictly convex inside $X$.

If $n=2$, it suffices to apply \cite[Theorem 2.19]{figalliBook} to deduce that $u$ is strictly convex inside $X$.
So we consider the case $n \geq 3$ and add the assumption that $F,1/F\in L^\infty(X\cap B_R)$ for any $R>0$.

The strategy of the proof is by now classical, although in the unbounded case the argument becomes much more delicate.

Assume by contradiction that $u$ is not strictly convex. Then there exist $\hat x\in X$ and $p \in \partial u(\hat x)$ such that the (convex set $\Sigma:=\{u=\ell\}$ is not a singleton, where $$\ell(z):=u(\hat x)+p\cdot(z-\hat x).$$
Note that $\Sigma=\{u\le \ell\}$ is closed, since $u$ is lower semicontinuous.

We now consider the following four cases, and we show that none of them can occur (cp. \cite[Proof of Theorem 4.23, Step 4]{figalliBook}, and see \cite[Section A.3.1]{figalliBook} for a definition of exposed points).
\begin{enumerate}
\item[$\bullet$] {\it $\Sigma$ has no exposed points in $\R^n$}. Indeed, otherwise $\Sigma$ would contain an infinite line, 
so \cite[Theorem A.10 and Lemma A.25]{figalliBook} imply that $\partial u(\R^n)$ is contained in
a hyperplane, contradicting \eqref{subgradimage}.

\item[$\bullet$] {\it $\Sigma$ has an exposed point inside $X$}. Since $u$ is finite inside $X$, if $\bar x\in X$ is an exposed point of $\Sigma$ then we can localize the problem near $\bar x$ as in the classical case (see \cite[Theorem 4.10]{figalliBook}). More precisely, since $\bar x\in X\subset {\rm dom}(u)$ and $u$ is locally Lipschitz inside its domain, we can find a neighborhood $\mathcal U_{\bar x}\subset X$ of $\bar x$ where $\nabla u$ is bounded.
Since by assumption $F,1/F\in L^\infty(X\cap B_R)$ and $G, 1/G\in L^\infty(Y\cap B_R)$ for any $R>0,$ it follows that $\frac{F}{G(\nabla u)}$ is bounded away from zero and infinity inside $\mathcal U_{\bar x}$. Hence, this allows us to repeat inside $\mathcal U_{\bar x}$ the very same argument as in the classical case  (see the proof of \cite[Theorem 4.10]{figalliBook})  to obtain a contradiction.

\item[$\bullet$] {\it $\Sigma$ has as exposed points in $\R^n\setminus \overline{X}$}.
Note that the Monge-Amp\`ere measure of $u$ vanishes outside $\overline X$ (see \eqref{MA outside X}). 

Assume that there exists an exposed point $\bar x\in \R^n\setminus \overline{X}$.
If $\bar x$ belongs to the interior of ${\rm dom}(u)$ then, as we did before, we can argue exactly as in the classical case to get the desired contradiction.
So, we can assume that $\bar x\in \Sigma
\cap \partial ({\rm dom}(u))$.
Since $\Sigma$ is convex and does not contain an infinite line, up to replacing $u$ by $u-\ell$ and up to a change of coordinates, we can assume that $\bar x$ coincides the origin $\zero$, $\ell=0$, $\Sigma\subset {\rm dom}(u)\subset \{z_1\leq 0\}$, and $\Sigma\cap \{z_1\geq -\eta\}$ is compact for any $\eta>0$. Also, since $\bar x \in \R^n\setminus \overline{X},$ if $\eta$ is small enough then $\Sigma\cap \{z_1\geq -\eta\}$ is disjoint from $\overline X$.

So, we fix $\eta_0$ small enough such that $\Sigma\cap \{z_1\geq -\eta_0\}\cap \overline X=\emptyset$,
and we
consider the functions
$$
u_\varepsilon(z):=u(z)-\varepsilon(z_1+\eta_0),\qquad S_\varepsilon:=\{u_\varepsilon<0\}.
$$
Note that $S_\varepsilon\to \Sigma\cap \{z_1\geq -\eta_0\}$ as $\varepsilon\to 0$,
therefore the sets $S_\varepsilon$ are bounded and $\overline S_\varepsilon\cap \overline X=\emptyset$ for $\varepsilon\ll 1$.
To get a contradiction, we would like to apply the so-called Alexandrov estimates (see \cite{caf2} or \cite[Theorem 4.23]{figalliBook} for more details). Usually, these estimates are stated for convex functions that are continuous up to the boundary, while in our case the graph  of $u_\varepsilon|_{\partial S_\varepsilon}$ contains some vertical segments in its graph (see Figure 1).
\begin{figure}[h]
\begin{center}
\includegraphics[scale=0.65]{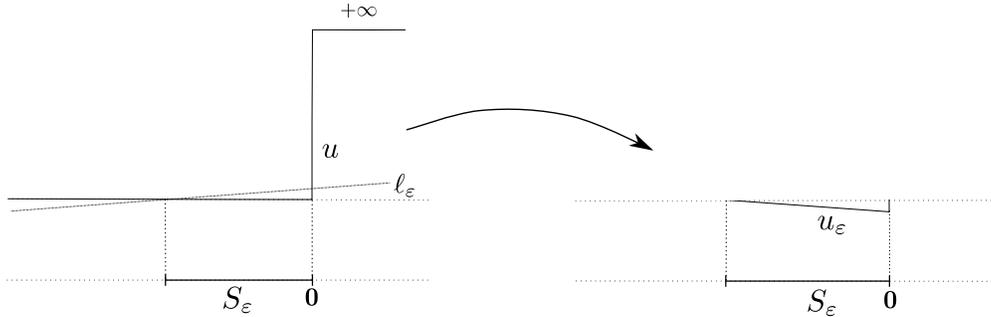}
\caption{We subtract the affine function $\ell_\varepsilon(z):=\varepsilon(z_1+\eta_0)$ from $u$. Because $\zero \in  \partial ({\rm dom}(u))$, $u_\varepsilon|_{\partial S_\varepsilon}$ contains some vertical segments in its graph.}
\end{center}
\label{fig:vert}
\end{figure} 
The key observation is that this does not create any problem, since one can approximate $u_\varepsilon$ with convex functions that are continuous in $\overline S_\varepsilon$, apply the usual Alexandrov estimates to these functions and then take the limit, to deduce that 
$$
0<\varepsilon \eta_0=|u_\varepsilon(\zero)|
\lesssim {\rm dist}(\zero,\partial S_\varepsilon)^{1/n}\bigl(\mathscr{M}_{u_\varepsilon}(\overline S_\varepsilon)\bigr)^{1/n}.
$$
However, since $\mathscr{M}_{u_\varepsilon}=\mathscr{M}_{u}$ (because $u_\varepsilon$ and $u$ differ only by an affine function) and 
$\overline S_\varepsilon\cap \overline X=\emptyset$, it follows by \eqref{MA outside X} that
$$
\mathscr{M}_{u_\varepsilon}(\overline{S_\varepsilon})=0,
$$
a contradiction.

\item[$\bullet$] {\it $\Sigma$ has as exposed points on $\partial X$}. 
As in the previous cases, the new difficulty arises if there exists an exposed point $\bar x\in \Sigma
\cap \partial ({\rm dom}(u))\cap \partial X$.
This is the case when a new argument and the extra assumptions appearing in the statement of the theorem are  needed.

\noindent
{\it - Case (2-a): $X=\R^n$.}
In this case $\partial X=\emptyset$, so there is nothing to prove.

\smallskip

\noindent
{\it - Case (2-b): $Y$ is bounded.}
Let $R>0$ be such that $Y\subset B_R(0)$.
Then, it follows by \eqref{subgradimage} and 
\eqref{extension} 
that $u$ can also be written as
$$
u(z) = \sup_{x\in X,\,  p \in \partial u (x),\,|p|\leq R} \big\{ u(x) + p\cdot (z-x)\big\}.
$$
This implies that $u$ is globally $R$-Lipschitz (being a supremum of $R$-Lipschitz functions), hence ${\rm dom}(u)=\R^n$ and therefore
$ \partial ({\rm dom}(u))=\emptyset$.
\smallskip

\noindent
{\it - Case (2-c): $X$ is locally uniformly convex.} By assumption,
for any $x \in \partial X$ there exists a small ball $B_r(x)$ where $\partial X$ is uniformly convex, namely, there exist $R>0$ and $x_0 \in \R^n$ such that $X\cap B_r(x)\subset B_R(x_0)$ and $x \in \partial B_R(x_0)$.

Since $\bar x\in \Sigma
\cap \partial ({\rm dom}(u))\cap \partial X$ and $X\subset {\rm dom}(u)\subset\{x_1\le 0\}$,
up to replacing $u$ by $u-\ell$ and up to a change of coordinates, we can assume that $\bar x=\zero$, $\ell=0$,
$$
u\geq tx_1\qquad \forall\,t\geq 0
$$
(see the shape of $u$ in Figure 1).
In particular $\partial u(\zero)$ contains the half-line $\R^+e_1$.

Consider the sets
$$
X_\theta:=X\cap \{x=(x_1,x')\in \R\times \R^{n-1}\,:\, -\theta |x'|\leq x_1 \leq 0\},
$$
and
$$
Y_\theta:=Y\cap \{y=(y_1,y')\in \R\times \R^{n-1}\,:\,y_1>0,\,|y'|\leq \theta y_1\}.
$$
We claim that
$$
X_\theta\supset(\nabla u)^{-1}(Y_\theta).
$$
Indeed, since $\R^+ e_1 \in \partial u(\zero)$, if $y=\nabla u(x) \in Y_\theta$ then it follows by monotonicity of $\partial u$ that
$$
(y-te_1)\cdot x =(\nabla u(x)-te_1)\cdot (x-\zero) \geq 0\qquad \forall\,t \geq 0.
$$
Setting $t=0$ this proves that $y\cdot x \geq 0$, while letting $t\to +\infty$ we get $x\cdot e_1\leq 0$.
Combining these two inequalities, one obtains the claim (cp. \cite[Figure 1]{FKM}).

Observe now that,
since $\R^+e_1\subset \partial u(\zero)\subset \partial u(\R^n)\subset \overline Y$, the $n$-dimensional convex set $\overline Y$ contains  both $\zero$ and $e_1$, so
a simple geometric argument shows that $ |Y_\theta\cap B_2| \gtrsim \theta^{n-1}$ for $\theta \ll 1$.\footnote{A possible proof is the following: note that the intersection of $Y_\theta$ with the hyperplane $\{y_1=1\}$ coincides with $Y_\theta^{n-1}:=Y\cap \{y_1=1\} \cap \{|y'|\leq \theta\}$, which has ($n-1$)-dimensional measure of order $\theta^{n-1}$.
Hence, by convexity, $\overline Y_\theta\cap B_2$ contains the cone generated by $\zero$ and $Y_\theta^{n-1}$. Noticing that the latter set has volume of order $\theta^{n-1}$, we conclude that $ |Y_\theta\cap B_2|=|\overline Y_\theta\cap B_2| \gtrsim \theta^{n-1}$.} 
Hence, by the transport condition we deduce that 
\begin{align*}
\int_{X_\theta}F(x)\,dx &\geq \int_{(\nabla u)^{-1}(Y_\theta)}F(x)\,dx=
\int_{Y_\theta}G(y)\,dy \\
&\geq \int_{Y_\theta\cap B_2}G(y)\,dy \gtrsim |Y_\theta\cap B_2| \gtrsim \theta^{n-1},
\end{align*}
where we used that $1/G\in L^\infty(Y\cap B_2).$
On the other hand, since $F\in L^\infty(X\cap B_R)$ and $X$ is uniformly convex near $\zero$, we have
$$
\int_{X_\theta}F(x)\,dx\lesssim|X_\theta|\lesssim  \theta^{n+1}
$$
(see Remark \ref{rmk:bdry} below for more details, and compare also with \cite[Figure 1]{FKM}).
Hence, we obtain a contradiction choosing $\theta$ sufficiently small.\\
\end{enumerate}
This concludes the proof of the strict convexity of $u$ inside $X$.
\\

\noindent
{\it - Step 5: conclusion.}
Let  us consider an arbitrary open ball $B$ with $\overline B\subset X\subset {\rm dom}(u)$. Then $\partial u(\overline B)$ is  a  compact set (see \cite[Lemma A.22]{figalliBook}),
and it follows by the strict convexity of $u$ that $\partial u(\overline B)$ is strictly contained inside $Y$.
Therefore,  the local boundedness assumption on the densities $F$ and $G$
(both from above and below) guarantees that there exist some constants $\lambda_B, \Lambda_B>0$ (depending on $B$) such that
 $$
 \lambda_B \le \det D^2 u \le \Lambda_B
\quad \textrm{ in the Alexandrov sense inside $B$.}
$$
Hence, the local regularity theory for strictly convex Alexandrov solutions applies 
(see \cite{cafC1a, cafC2a,dPF,dPFS,FJM}, and also \cite[Chapter 4.6.1]{figalliBook}),
and we deduce that 
$u \in C^{1,\alpha}_{\rm loc}(B)\cap W^{2,1+\epsilon}_{\rm loc}(B)$.
In particular, since $B$ was arbitrary, $u \in C^1(X)$.

Thanks to the strict convexity and $C^1$ regularity of $u$, we deduce that $\nabla u:X\to Y$ is a homeomorphism onto its image. Also, recalling \eqref{subgradimage}, $\nabla u(X)$ is a open subset of $Y$ of full measure inside $Y$.
Finally, if $X=\R^n$ then $\nabla u(\R^n)$ is a convex set (see \cite{Gr}), therefore $\nabla u(X)=Y$.
This concludes the proof of (i).

To prove (ii) we note that if $k\geq 0$, $\beta \in (0,1)$,  and $F$ and $G$ are of class $C^{k,\beta}_{\rm loc}$ on $X$ and $Y$ respectively, then $u\in C^{k+2,\beta}_{\rm loc}(X)$ (see for instance \cite[Chapter 4.6.1]{figalliBook}).
Hence, since $\det D^2 u>0$ inside $X$, it follows that $T=\nabla u:X\to Y$ is a $C^{k+1,\beta}_{\rm loc}$ diffeomorphism onto its image.
\end{proof}

\begin{rem}
{\em As the reader may have noticed, the first three steps of the proof of Theorem~\ref{maintheo} require weaker assumptions and show the following fact:\\}
Let $X$ and $Y$ be two open sets in $\R^n$,  and assume that $|\partial X|=0$
and that $Y$ is convex. Let $F,G:\R^n\to \R$ be two nonnegative Borel functions with $\int_X F = \int_Y G= 1$ such that $F,1/F\in L^\infty_{\rm loc}(X)$ and $G, 1/G\in L^\infty_{\rm loc}(Y)$. Denote by $T=\nabla u$ the Brenier-McCann map  pushing $1_X(x) F(x)dx$ forward to $1_Y(y)G(y)dy$. Then $u$ is an Alexandrov solution 
(namely, Equation \eqref{eq:Alex} holds).
In particular, if in addition $G, 1/G\in L^\infty(Y\cap B_R)$ for any $R>0$ then (arguing as in the proof proof of Theorem~\ref{maintheo}) one deduces that $\mathscr M_u\in L^\infty_{\rm loc}(X)$, and therefore 
$u \in W^{2,1}_{\rm loc}(X)$ thanks to \cite{mooney}.
\end{rem}

\begin{rem}\label{rmk:bdry}
{\em As already mentioned in Remark \ref{rmk:unif convex}, the uniform convexity of $X$ in Case (2-c) can be weakened.
More precisely, as shown in Step 4 of the proof above,
if we choose a system of coordinates such that $$X\subset \{x=(x_1,x')\in \R\times \R^{n-1}\,:\, x_1 \leq 0\},\qquad \zero \in \partial X,$$
then it suffices to ensure that 
$$
X_\theta:=X\cap \{x=(x_1,x')\in \R\times \R^{n-1}\,:\, -\theta |x'|\leq x_1 \leq 0\},
$$ 
satisfies $|X_\theta|=o(\theta^{n-1})$ as $\theta \to 0$. To ensure this, assume that $\partial X$ can be written locally as the graph of a nonnegative function $f:\R^{n-1}\to \R$ with $-f(x')\leq -c|x'|^\lambda$ for some $c>0$ and $\lambda>0$. Note that, since $|x'|$ is small, with no loss of generality we can assume that $\lambda>1$.
Then
\begin{multline*}
X_\theta = \{x=(x_1,x')\in \R\times \R^{n-1}\,:\, -\theta|x'|\leq x_1\leq -f(x')\}\\
\subset
\{x=(x_1,x')\in \R\times \R^{n-1}\,:\, -\theta|x'|\leq x_1\leq -c|x'|^\lambda\}.
\end{multline*}
Note that the inequality $c|x'|^\lambda \leq \theta |x'|$ implies that $|x'|\lesssim \theta^{1/(\lambda-1)}$, therefore $|x_1|\leq \theta |x'| \lesssim \theta^{\lambda/(\lambda-1)}$.
Thus
$$
|X_\theta| \lesssim \theta^{\frac{\lambda}{\lambda-1}}\theta^{\frac{n-1}{\lambda-1}}=\theta^{\frac{n-1+\lambda}{\lambda-1}}.
$$
In particular, to have $|X_\theta|=o(\theta^{n-1})$ we need to ensure that
$$
\frac{n-1+\lambda}{\lambda-1}>n-1 \qquad \Longleftrightarrow\qquad \lambda<\frac{2(n-1)}{n-2}.
$$
}
\end{rem}

\appendix

\section{A counterexample}
In this appendix  we 
show that, given $Y=\R^n$ with $n\geq 3$, one can find a convex set $X$
and 
two nonnegative Borel functions with $\int_X F = \int_Y G= 1$ such that $F,1/F$ belong to $L^\infty_{\rm loc}(X)$ and $G, 1/G$ belong to $L^\infty_{\rm loc}(Y)$, but for which $u$ is not strictly convex nor of class $C^2$. Our example is inspired by the classical Pogorelov's counterexample, see \cite[Section 3.2]{figalliBook}. 

\smallskip
 
Let $n\geq 3$, and consider the function 
defined on $X:=\R^{n-1}\times (-1,1)\to \R$ as 
$$
u(x',x_n):=|x'|^\gamma h(x_n),\qquad \gamma:=2-\frac2n>1, 
$$
with $h(x_n):=(1-x_n^2)^{-\alpha}$ for some suitable $\alpha>0$ to be chosen.
Since
$$
\det D^2 u=\gamma^{n-1}h^{n-2}[(\gamma-1)h h'' - \gamma (h')^2]
$$
and
$$
h'=\frac{2\alpha x_n}{(1-x_n^2)^{\alpha+1}},\qquad 
h''=\frac{2\alpha +2\alpha x_n^2 + 4 \alpha^2x_n^2}{(1-x_n^2)^{\alpha+2}},
$$
we see that $\det D^2u>0$ provided
$$
(\gamma-1)[2\alpha +2\alpha x_n^2 + 4 \alpha^2x_n^2]>4\gamma \alpha^2 x_n^2, 
$$
or equivalently
$$
(\gamma-1)(1+x_n^2)>2\alpha x_n^2.
$$
Since $2x_n^2<1+x_n^2$ for $|x_n|<1$, a sufficient condition is to impose $\alpha\leq \gamma-1$.
As we shall see later, we actually need to impose $\alpha<\gamma-1$.

Note that, by the particular structure of $D^2u$ (see \cite[Section 3.2]{figalliBook}) the condition $\det D^2u>0$ ensures also that $u$ is convex. To be precise, $u$ is defined on $\dom(u) = X \cup \{(0,1)\}\cup\{(0,-1)\}$ with value $0$ at these two extra points, so that $u$ is  lower semicontinuous on $\R^n$.

We now want to look at the image of $\nabla u$.

Observe that
\begin{equation}
\label{eq:Du}
\nabla u(x',x_n)=\biggl(\frac{\gamma |x'|^{\gamma-2}x'}{(1-x_n^2)^\alpha},\frac{2\alpha |x'|^\gamma x_n}{(1-x_n^2)^{\alpha+1}}\biggr).
\end{equation}
We claim that $\nabla u(X)=\hat Y \cup\{0\}$ where $\hat Y :=(\R^{n-1}\setminus \{0\})\times \R.
$

To show this, consider a point $y\in \hat Y$, and let us look for a point $x$ such that 
$
y=\nabla u(x).
$
First of all, it follows by \eqref{eq:Du} that $y'$ and $x'$ must be parallel in $\R^{n-1}$, while $y_n$ and $x_n$ have the same sign. Also, using \eqref{eq:Du} again, if $y_n\neq 0$ we get
\begin{equation}
\label{eq:y x}
\frac{|y'|^{\frac{\gamma}{\gamma-1}}}{|y_n|}=
\frac{\gamma^{\frac{\gamma}{\gamma-1}}(1-x_n^2)^{\alpha+1}}{2\alpha (1-x_n^2)^{\alpha\frac{\gamma}{\gamma-1}}|x_n|}=\frac{\gamma^{\frac{\gamma}{\gamma-1}}(1-x_n^2)^{\frac{\gamma-\alpha-1}{\gamma-1}}}{2\alpha |x_n|}.
\end{equation}
Note that, because $\alpha<\gamma-1$, the numerator of the expression in the right hand side goes to zero as $|x_n|\to 1^-$ while the denominator goes to infinity as $|x_n|\to 0$.
This proves that the function
$$
x_n\mapsto \frac{\gamma^{\frac{\gamma}{\gamma-1}}(1-x_n^2)^{\frac{\gamma-\alpha-1}{\gamma-1}}}{2\alpha |x_n|}
$$
is surjective from $(0,1)$ onto $(0,+\infty)$.
In other words, given $y=(y',y_n)$,
we can use the relation \eqref{eq:y x}
to find $x_n$ (the sign of $x_n$ is given by the one of $y_n$, while $x_n=0$ if $y_n=0$), and then use
the relation
$$
y'=\frac{\gamma |x'|^{\gamma-2}x'}{(1-x_n^2)^\alpha}
$$
(see \eqref{eq:Du}) to find $x'$.
This proves the claim.

Also, since $\det D^2u>0$ and $u$ is smooth for $|x'|>0$, it is easy to check that $\nabla u$ is a diffeomorphism from $\hat X:=(\R^{n-1}\setminus \{0\})\times (-1,1)$ onto $\hat Y$. 

Now, note that $Y$ differs from $\R^n$ by a set of measure zero.
Hence, we can define $G(y):=c_ne^{-|y|^2/2}$ on $\R^n$,
and set
$$
F(x):=\det D^2 u(x) \,G(\nabla u(x))\qquad \forall \,x\in X.
$$
Since $\nabla u$ maps compact subset of $X$ onto compact subsets of $Y$,
it follows that
 $F,1/F \in L^\infty_{\rm loc}(X)$, but $u$ is not strictly convex and it is not of class $C^2$. 

The interested reader may observe that, in this example,
 $F(x)\to 0$ as $x\to \partial X \setminus\{\pm e_n\}$.
Note that this is essentially needed: thanks to Corollary \ref{maincor}, one cannot construct a similar example with $F,1/F\in L^\infty (X\cap B_R)$ for any $R>0$.

\begin{rem}
{\em Let us summarize, from a convexity viewpoint, the example above. The lower semicontinuous convex function $u$ has domain  $\R^{n-1}\times (-1,1)\cup \{(0,\pm 1)\}\subset \R^n$ and tends to infinity at the boundary except at the two exceptional points $\{(0,\pm 1)\}$,  the function being constant equal to zero on the segment $\{0\}\times[-1,1]$. The function $u$ is $C^1$ on the open set  $\R^{n-1}\times (-1,1)$  and the image of the gradient is $[(\R^{n-1}\setminus \{0\})\times \R] \cup \{0\}$. At the two remaining points $\{(0,1)\}$ and $\{(0,-1)\}$  on the boundary, $u$ has as subgradients the half-lines  $\{0\}\times (0,+\infty)$ and $\{0\}\times (-\infty,0)$, respectively.  We can further point out that the Legendre transform $v=u^\ast$ has domain $\R^n$ and it  is differentiable on $\R^{n}\setminus \{0\}$, with
$$ \partial v(0) = \{0\} \times [-1, 1] \qquad \textrm{and} \qquad   \nabla v (0, \pm s) = \{0\}\times \{\pm 1\}\  \textrm{ for every $s>0$.}
$$ 
Note that some points in the interior of the domain of $v$  are sent to points in the boundary of the domain of $u$, a situation that is extremal with respect to monotonicity of (sub)gradients. }
\end{rem}
\bigskip

{\it Acknowledgments:} A. Figalli  has received funding from the European Research Council under the Grant Agreement No. 721675 ``Regularity and Stability in Partial Differential Equations (RSPDE)''.

\end{document}